\documentclass[12pt]{article} 

\usepackage{amsmath}
\usepackage{amsthm}
\usepackage{amsfonts}
\usepackage{mathrsfs}
\usepackage{stmaryrd}
\usepackage{setspace}
\usepackage{fullpage}
\usepackage{amssymb}
\usepackage{breqn}
\usepackage{enumitem}
\usepackage{bbold} 
\usepackage{authblk}
\usepackage{comment}
\usepackage{hyperref}
\usepackage{pgf,tikz}
\usepackage{graphicx}
\usetikzlibrary{decorations.pathreplacing,arrows}
\tikzstyle{vertex}=[circle,draw=black,fill=black,inner sep=0,minimum size=3pt,text=white,font=\footnotesize]

\newtheorem{thm}{Theorem}[section]%[chapter]

\newtheorem{prop}[thm]{Proposition}

\newtheorem{defin}[thm]{Definition}

\newtheorem{corollary}[thm]{Corollary}

\newtheorem{clm}[thm]{Claim}
\newtheorem{observation}[thm]{Observation}
\newtheorem*{lemma*}{Lemma}
\newtheorem*{proposition*}{Proposition}
\newtheorem*{theorem*}{Theorem}

\newcommand\ex{\ensuremath{\mathrm{ex}}}
\newcommand\re{\ensuremath{\mathrm{na}}}

\newcommand\cC{{\mathcal C}}

\newcommand\cF{{\mathcal F}}

\newcommand\cN{{\mathcal N}}

\newcommand{\ignore}[1]{}

\title{A non-aligning variant of generalized Tur\'an problems}
\author{D\'aniel Gerbner\footnote{Alfr\'ed R\'enyi Institute of Mathematics, E-mail: \texttt{gerbner.daniel@renyi.hu.} Research supported by the
    National Research, Development and Innovation Office -- NKFIH under the
    grants FK 132060, KKP-133819, KH130371 and SNN 129364.}}
    
    \date{}

\begin{document}

\maketitle

\begin{abstract}
In the so-called generalized Tur\'an problems we study %$\ex(n,H,F)$, which is 
the largest number of copies of $H$ in an $n$-vertex $F$-free graph $G$.
   Here we introduce a variant, where $F$ is not forbidden, but we restrict how copies of $H$ and $F$ can be placed in $G$. More precisely, given an integer $n$ and graphs $H$ and $F$, what is the largest number of copies of $H$ in an $n$-vertex graph such that the vertex set of that copy does not contain and is not contained in the vertex set of a copy of $F$?

    We solve this problem for some instances, give bounds in other instances, and we use our results to 
    %show that $\ex(n,K_{a,b},F)$ is... nincs $\cN$ definiálva... for every 3-chromatic graph $F$ with a color-critical edge nem igazán 
    determine the generalized Tur\'an number for some pairs of graphs.
\end{abstract}

\section{Introduction}

One of the most basic question of extremal graph theory is due to Tur\'an \cite{T}: given a graph $F$ and an integer $n$, what is the largest number of edges an $n$-vertex graph can have if it does not contain a subgraph isomorphic to $F$? This quantity is denoted by $\ex(n,F)$ and is a subject of extensive research. In particular, Tur\'an determined $\ex(n,K_{k+1})$ for every $n$ and $k$. It is the number of edges in the graph $T_k(n)$, which is the $n$-vertex complete $k$-partite graph with each part having order $\lfloor n/k\rfloor$ or $\lceil n/k\rceil$. $T_k(n)$ is now called the \emph{Tur\'an graph}.

A natural generalization is the following. Instead of the edges, we count the copies of another subgraph $H$. We denote by $\ex(n,H,F)$ the largest number of copies of $H$ in $n$-vertex $F$-free graphs. Problems concerning this quantity are often simply called \emph{generalized Tur\'an problems}. For graphs $H$ and $G$, let $\cN(H,G)$ denote the number of copies of $H$ in $G$. Then $\ex(n,H,F)=\max\{\cN(H,G): G \text{ has $n$ vertices and is $F$-free}\}$. If for an $n$-vertex graph $G$ we have $\ex(n,H,F)=\cN(H,G)$, then we say that $G$ is an \emph{extremal graph} with respect to $H$ and $F$ (we simply say that $G$ is extremal if $H$ and $F$ are clear from the context).

The first result concerning $\ex(n,H,F)$ is due to Zykov \cite{zykov}, who showed that the Tur\'an graph is extremal with respect to $K_k$ and $K_r$. After several sporadic results, the systematic study of $\ex(n,H,F)$ was initiated by Alon and Shikhelman \cite{as}.

\medskip

Here we study a variant of $\ex(n,H,F)$, where $F$ is not forbidden, but we only count copies of $H$ that are far from copies of $F$ in the following sense.

\begin{defin}
We say that two subgraphs $G'$ and $G''$ of a graph $G$ \textit{align} in $G$ if $V(G')\subseteq V(G'')$ or $V(G'')\subseteq V(G')$.
\end{defin}

\begin{defin}\label{main}
Given graphs $H$, $F$ and $G$, we denote by $\cN_{\re}(H,F,G)$ the number of copies of $H$ in $G$ that do not align with any copy of $F$ in $G$. When $F$ is clear from the context, we say that a copy of $H$ is nice in $G$ if it does not align with any copy of $F$ in $G$.
Given graphs $H$ and $F$, we let $\ex_{\re}(n,H,F):=\max\{\cN_{\re}(H,F,G): \text{ $G$ is an $n$-vertex graph}\}$. 
%the largest number of copies of $H$ 
\end{defin}

By definition, $\ex_{\re}(n,H,F)\ge \ex(n,H,F)$. We remark that the situation is very different depending on whether $H$ or $F$ has more vertices. When looking at a copy of $H$ in $G$, we either have to examine the subgraphs or the supergraphs of it to decide if it aligns with a copy of $F$. 
%Perhaps one could cut Definition \ref{main} to two parts that coincide when $|V(H)|=|V(F)|$. However, in some cases we can handle them together.
%Our main reason to combine them into one definition is avoid introducing too many terms.

The study of the case $|V(H)>|V(F)|$ was suggested by the author in \cite{gerbner}. The reason is that counting copies of $H$ that do not align with any $F$ was used there to prove a generalized Tur\'an result. Let us describe now the general approach. 
%Assume we want to prove an upper bound on $\ex(n,H,F)$ and have upper bounds on $\ex(n,H,F')$ and $\ex(n,F',F)$. 
Observe that an $n$-vertex $F$-free graph contains at most $\binom{n}{|V(H)|-|V(F')|}\ex(n,F',F)$ copies of $H$ that align with a copy of $F'$. Thus an upper bound on $\ex_{\re}(n,H,F')$ implies an upper bound on $\ex(n,H,F)$.

\medskip

Let us continue with a connection to another well-studied notion in extremal graph theory, namely inducibility. Let us denote by $i(H,n)$ the largest number of \emph{induced} copies of a graph $H$ in any $n$-vertex graph. If $|V(H)|\ge |V(F)|$ and $H$ does not contain $F$, then induced copies of $H$ do not align with any copy of $F$, thus $\ex_{\re}(n,H,F)\ge i(F,n)$. In particular, this shows the following. 
\begin{observation}
If $|V(H)|\ge |V(F)|$ and $H$ does not contain $F$, then we have $\ex_{\re}(n,H,F)=\Theta(n^{|V(H)|})$.
\end{observation}

We remark that this behaviour is very different from $\ex(n,H,F)$. 

Let us assume now that $|V(H)|< |V(F)|$. Let $\cF$ denote the family of graphs obtained by adding edges (but no vertices) to $F$. If no member of $\cF$ contains $H$ as an induced subgraph, then induced copies of $H$ do not align with any copy of $F$, thus $\ex_{\re}(n,H,F)\ge i(F,n)$.
In the remaining cases when $H$ is an induced subgraph of some $F'\in \cF$, the order of magnitude of $\ex_{\re}(n,H,F)$ can be different from both $n^{|V(H)|}$ and $\ex(n,H,F)$, as shown by the following example. We denote by $P_k$ the path with $k$ vertices.

\begin{prop}\label{utak}
Let $3<k<\ell$ and $\ell-k$ odd. Then $\ex_{\re}(n,P_k,P_\ell)=\Theta(n^{k-2})$.
\end{prop}
The case where $\ell-k$ even is very different, as we discuss in Section \ref{concl}.
We note that for any $k$ and $\ell$, $\ex(n,P_k,P_\ell)=\Theta(n^{\lceil k/2\rceil})$ by a result of Gy\H ori, Salia, Tompkins and Zamora \cite{gstz} (who also determined the asymptotics, and in some cases they obtained exact results).

Let us return to the connection to inducibility. In some cases, every copy of $H$ that does not align with any $F$ must be induced. It is the case when $H$ is a complete $k$-partite graph (but not $K_k$) and $F=K_{k+1}$, or more generally $F$ can be obtained by adding an edge to a complete $k$-partite graph $H'$ such that the largest partite set of $H'$ has no more vertices than the smallest partite set of $H$. 

Brown and Sidorenko \cite{bs} proved the following.

\begin{thm}[Brown, Sidorenko, \cite{bs}]\label{brsi}
\textbf{(i)} If $H$ is complete multipartite, then the $n$-vertex graph containing the most induced copies of $H$ is also complete multipartite.

\textbf{(ii)} If $H$ is complete bipartite, then the $n$-vertex graph containing the most induced copies of $H$ is also complete bipartite. Moreover, if $G$ is a complete multipartite $n$-vertex graph but not bipartite, then $G$ contains at most $i(H,n)-\Theta(n^{|V(H)|})$ induced copies of $H$.
\end{thm}

We remark that they do not state the moreover part of $\textbf{(ii)}$, but it follows from their proof and we will use it later. This result implies that $\ex_{\re}(n,H,K_{k+1})$ is given by some complete multipartite graph if $H$ is complete $k$-partite but not $K_k$. In particular $\ex_{\re}(n,K_{a,b},K_3)=\cN(K,T)$ for some complete bipartite $n$-vertex graph $T$ (this clearly holds also when $K_{a,b}$ is the complete graph $K_2$). The missing case $\ex_{\re}(n,K_k,K_{k+1})$ is also known: it is the largest number of maximal cliques of size $k$. Byskov \cite{byskov} showed that the Tur\'an graph $T_k(n)$ contains the most such cliques.

We can extend the above results to $\ex_{\re}(n,H,K_{k+1})$ if $H$ is complete $k'$-partite for any $k'$.

\begin{thm}\label{zyk}
Let $K$ be a complete multipartite graph. Then $\ex_{\re}(n,K,K_{k+1})=\cN_{\re}(K,K_{k+1},T)$ for a complete $r$-partite graph $T$ for some $r\ge k$. 
%Moreover, if $r$ is the largest integer satisfying the above property and $G$ has chromatic number more than $r$, then $\cN_{\re}(H,F,G)\le \cN(K,T)-\Omega(n^{|V(K)|-1}$. ja nem, ha pl van plusz egy él, az csak $\Omega(n^{|V(K)|-2}$ példányt tesz tönkre...
\end{thm}

We remark that $r$ does not have to be equal to $k$, similarly to Theorem \ref{brsi}. We will prove a more general version where we count multiple graphs, with weights, at the same time. We will use it later.

\begin{thm}\label{zyk2}
Let $H_1,\dots,H_\ell,H'_1,\dots,H'_m$ be complete multipartite graphs, $\alpha_i\ge 0$ 
%unless $H_i$ is a clique 
and $\beta_j\ge 0$. 
%unless $H'_j$ is a clique. 
Then among $n$-vertex graphs $G$, $x(G):=\sum_{i=1}^\ell \alpha_i\cN_{\re}(H_i,K_{k_i+1},G)+\sum_{j=1}^m\beta_j \cN(H'_j,G)$ is maximized by a complete multipartite graph $T$. 
%Moreover, if hmm, ezt már nehéz ilyen általánosan, vegyük a legnagyobb $x$ pontszámút aminek pozitív az együtthatója, és $n^{x-2}$ amit bukunk ha nem complete multipartite...
\end{thm}

Clearly, Theorem \ref{zyk} is implied by Theorem \ref{zyk2}, except that in the former theorem we state that $T$ has at least $k$ parts. But that is trivial, as if $T$ had less than $k$ parts, we could make it $k$-partite by adding edges without creating any copy of $K_{k+1}$.

Let us return to generalized Tur\'an problems. Gy\H ori, Pach and Simonovits \cite{gypl} showed that if $K$ is complete multipartite, then $\ex(n,K,K_{k+1})=\cN(K,T)$ for some complete $k$-partite graph $T$. Schelp and Thomasson \cite{scth} (extending a result of Bollob\'as \cite{boll}) showed that if we add up different complete multipartite graphs with weights, where the weight is non-negative unless the corresponding graph is complete, then again a complete multipartite graph is extremal. See \cite{gerbner3} for more generalized Tur\'an results where we count multiple graphs.

The proofs of most of the results mentioned above use the same technique, Zykov symmetrization \cite{zykov}, that we will also use in the proof of Theorem \ref{zyk2}. This makes the graph more symmetric (and eventually complete multipartite) without creating copies of $K_k$. Thus we cannot use it in the case of other forbidden graphs $F$. 

We know that in generalized Tur\'an problems, forbidding $K_k$ or forbidding any $k$-chromatic graph is closely related by the following result.
\begin{thm}[Gerbner and Palmer \cite{gp2}]\label{gepa}
Let $H$ be a graph and $F$ be a graph with chromatic number $k$, then $\ex(n,H,F)\le \ex(n,H,K_k)+o(n^{|V(H)|})$.
\end{thm}

If $F$ has chromatic number $k$ and $K$ has chromatic number less than $k$, then Theorem \ref{gepa} gives an asymptotic result for $\ex(n,K,F)$. Theorem \ref{zyk} gives an improvement in the error term in some cases: if $\ex_{\re}(n,K,K_{k+1})=\ex(n,K,K_{k+1})$ (which is the case if $k=2$), then we have $\ex(n,K,F)\le \ex(n,K,K_{k+1})+O(n^{|V(K)|-k-1})\ex(n,K_{k+1},F)$. Indeed, in an $F$-free graph, we can count the copies of $K$ that align with a $K_{k+1}$ by picking a $K_{k+1}$ and then picking the other $|V(K)|-k-1$ vertices. For some $F$, this gives an improvement in the error term.

We can use Theorem \ref{zyk2} to give an exact bound in some cases.
We say that an edge $e$ of a graph is color-critical, if by deleting $e$ we obtain a graph with smaller chromatic number. In extremal problems, $k$-chromatic graphs with a color-critical edge often behave very similarly to $K_k$.

\begin{thm}\label{getu} Let $F$ be a 3-chromatic graph with a color-critical edge and $a,b\ge 2$. Then for large enough $n$ we have that
$\ex(n,K_{a,b},F)=\ex(n,K_{a,b},K_3)=\cN(K_{a,b},K_{t,n-t})$ for some $t$.
\end{thm}

Note that in the case both $a$ and $b$ are large enough, in $F$-free graphs every copy of $K_{a,b}$ is induced, thus Theorem \ref{brsi} implies the above theorem in that case. We have mentioned that in \cite{gerbner}, a result that fits in our setting was used to prove a sharp bound for a generalized Tur\'an problem. In more details, $F$ was a $(k+1)$-chromatic graph with a color-critical edge such that $\ex(n,K_{k+1},F)=o(n^k)$. $H$ was a member of a large class of graphs, that includes very balanced complete $k$-partite graphs, and the extremal graph was $T_k(n)$. For those graphs $H$ in the case $k=2$, and any 3-chromatic graph $F$ with a color-critical edge (without any assumption on $\ex(n,K_{k+1},F)$), Theorem \ref{getu} shows that $T_2(n)$ is the extremal graph.

%gives an improvement: we have the same extremal graph for any  (we do not need to assume anything on $\ex(n,K_{k+1},F)$).

%ez csalás, ehhez nem kellett volna bevezetni a relatívot, elég lett volna induced.
%Ha $K$ $r$-partite és minden part legfeljebb 2 méretű, akkor ugyanígy, amíg $k+1\le 2r$, addig csak valami kompl multipartitenek az induced példányait tiltjuk. Ha $k+1$ nagyobb, akkor persze más a helyzet, mert nem a $K$-n belüli részek számítanak. Ha meg van $\ge 3$-as rész, abba berakott él tiltva van induced esetben, de itt meg nem. Mégis, kéne $K_{2,2}$ and $K_{2,3}$ vs 4-chrom a small graph II cikk miatt... talán $K_3$ a közbülső gráf? Ha ugyanazt akarjuk: kéne hogy Brown-Sidorenkoban nem csak páros jobb a többinél, hanem monoton strictly rosszabbodnak, és akkor jók vagyunk ha kicsi súlyt kapnak a $K_4$-ek. És pont azt csinálják!!! sajnos $K_{1,1,2}$-re már a teljes 5-részes az opti inducibilityre. Tehát amik 2 részbe metszenek, azok egyre rosszabbak, de amik meg 3 részbe, azok nőhetnek 5-ig.

\smallskip

Let us show a simpler example where $\ex_{\re}(n,H,F)=\ex(n,H,F)$. Gerbner and Patk\'os \cite{gepat} showed that if $2<t<t'$, then $\ex(n,K_{2,t},K_{2,t'})=(1+o(1))\binom{t'-1}{t}\binom{n}{2}$. The upper bound is proved by taking 2 vertices $u$ and $v$, $\binom{n}{2}$ ways, and  then counting the copies of $K_{2,t}$ where $u$ and $v$ form the smaller part. As they have at most $t'-1$ common neighbors, there are at most $\binom{t'-1}{t}$ such copies of $K_{2,t}$. In our case, we can proceed the same way: if $u$ and $v$ have at most  $t'-1$ common neighbors, we get the same upper bound. If $u$ and $v$ have at least $t'$ common neighbors, then every copy of $K_{2,t}$ where $u$ and $v$ form the smaller part aligns with a $K_{2,t'}$, thus we get the upper bound 0.

One of the most studied generalized Tur\'an problems is $\ex(n,K_3,C_{\ell})$. Let $\cC_\ell=\{C_3,C_4,\dots,C_\ell\}$. Gy\H ori and Li \cite{LiGy} showed that $\ex(n,K_3,C_{2k+1})=O(\ex(n,C_{2k}))$ and $\ex(n,K_3,C_{2k+1})=\Omega(\ex(n,\cC_{2k}))$ (their proof also shows the same bounds for $\ex(n,C_3,C_{2k})$). A theorem of Bondy and Simonovits \cite{bosi} states that $\ex(n,C_{2k})=O(n^{1+1/k})$. A conjecture of Erd\H os and Simonovits \cite{ersim} states that this bound is sharp, moreover $\ex(n,\cC_{2k})=\Theta(n^{1+1/k})$. This is known only in the cases $k=2,3,5$, see e.g. \cite{luw2}.

We can extend the result of Gy\H ori and Li to our setting.

\begin{prop}\label{korok} For any $k\ge 2$ we have that
$\ex_{\re}(n,K_3,C_{2k+1})=O(n^{1+1/k})$.
%and $\ex_{\re}(n,K_3,C_{2k+1})=\Omega(\ex(n,\cC_{2k}))$.
\end{prop}

Surprisingly, the case of even cycles is more complicated. One can easily show (and we will show in the proofs) that the number of triangles in a $C_\ell$-free graph $G$ is at most $(\ell-3)|E(G)|/3$, and similarly, the number of triangles not aligning with any $C_\ell$ in an arbitrary graph $G'$  is at most $(\ell-3)|E(G')|/3$. In the case $\ell$ is even, this immediately shows that $\ex(n,K_3,C_{2k})=O(\ex(n,C_{2k}))$. However, $G'$ can have more than $\ex(n,C_{2k})$ edges, and indeed the same upper bound does not hold.

\begin{thm}\label{parkor} \textbf{(i)} $n^{2-o(1)}=\ex_{\re}(n,K_3,C_4)=o(n^2)$.

\textbf{(ii)} For any $k\ge 3$ we have that
$\ex_{\re}(n,K_3,C_{2k})=O(n^{1+\frac{1}{k-1}})$ and $\ex_{\re}(n,K_3,C_{2k})=\Omega(\ex(n,\cC_{2k-2}))$.
\end{thm}

The rest of the paper is organized as follows. Section 2 contains the proofs of Propositions \ref{utak} and \ref{korok} and of Theorems \ref{zyk2}, \ref{getu} and \ref{parkor}. We finish the paper with some concluding remarks in Section 3.

\section{Proofs}

Let us prove Proposition \ref{utak}, that we restate here for convenience.

\begin{proposition*}
Let $3<k<\ell$ and $\ell-k$ odd. Then $\ex_{\re}(n,P_k,P_\ell)=\Theta(n^{k-2})$.
\end{proposition*}

\begin{proof}
 The lower bound is given by the following construction: we take a copy of $P_k$ with vertices $v_1,\dots,v_k$ in this order, and blow up each vertex except for $v_2$ and $v_{k-1}$ to linear size. More precisely, for $1\le i\le k$, $2\neq i\neq k-1$ we add $v_i^j$ for $j\le n/k$ and connect each $v_i^j$ to $v_{i-1}$, to each $v_{i-1}^m$, to $v_{i+1}$ and to each $v_{i+1}^m$. Then we obtain $\Omega(n^{k-2})$ copies of $P_k$ with vertices $v_1^{m_1},v_2,v_3^{m_3}\dots, v_{k-2}^{m_{k-2}},v_{k-1},v_k^{m_k}$ in this order. As the first and last vertices of these paths have degree 1, these paths cannot be extended to a longer path. 
 
 The only way such a path $P$ can align with a copy $P'$ of $P_\ell$ is if the endpoints of $P'$ are $v_1^{m_1}$ and $v_k^{m_1}$ and $P'$ leaves $P$ and returns to $P$ at least once. Then $P'$ visits the vertices of $P$ in an order $u_1,u_2,\dots,u_{k-1},u_k$, where $u_1=v_1^{m_1}$
 %, $u_2=v_2$, $u_{k-1}=v_{k-1}$ 
 and $u_k=v_k^{m_k}$. Let $f(i)$ denote the number of vertices on $P$ between $u_i$ and $u_{i+1}$, and $f'(i)$ denote the number of vertices on $P'$ between $u_i$ and $u_{i+1}$ (including $u_i$ and $u_{i+1}$). Observe that $f(i)$ and $f'(i)$ have the same parity, because the construction is a bipartite graph. Observe that $\sum_{i=1}^{k-1} f(i)$ is even, as we count every vertex but $u_1$ and $u_k$ an even number of times: if we jump over it, we have to jump back. This implies that $P'$ has an even number of vertices not in $P$, a contradiction to our assumption that $\ell-k$ is odd.

% hmm, ez nem igaz, azok az elso es utolso, de kozben kitérhetunk pl. $v_1v_2v_3v_4v_3^1v_4^1v_5$... $P_k$ vs $P_{k+1}$-re jó... vagy ha $k\le 5$. Egyébként meg talán $n^{k/2}$ kb? mert nem fújhatunk fel szomszédosokat és mert minden második élre lin ways, ha nem akkor ott hosszú út mint lejjebb. Gond hogy ne ütközzön bele az út további részébe... ja azt előbb rögzítjük. Na de hogy választjuk pl az első élt?...Ja ha $\ell-k$ odd, akkor ez jó. Ha even: minden második élt nézve nagyjából $C_{k-\ell}$-mentesnek kell lennie, tehát $\ex(n,C_{k-\ell}^{\lfloor k/2\rfloor}$, és még -szor $n$ ha $k$ páros. Alulról: vehetünk ilyet, és aztán egy pontot, tehát 3-anként van $n\ex(n,C_{k-\ell})$-es szorzó

 For the upper bound, let $G$ be an $n$-vertex graph, and let us pick a copy of $P_k$  with vertices $v_1,\dots,v_k$ in this order such that this copy does not align with any $P_\ell$, the following way.
 First we pick the vertices different from $v_1,v_2,v_{k-1},v_k$, and then the edges $v_1v_2$ and $v_{k-1}v_k$. There are $O(n^{k-4})$ ways to pick the vertices, so if there are at most $\ell n$ ways to pick both the edges $v_1v_2$ and $v_{k-1}v_k$, then we are done.
 
 Assume without loss of generality that there are more than $\ell n$ choices for $v_1v_2$, let $E_1$ be the set of these edges and $G_1$ be the graph on $V(G)$ and edge set $E_1$. Then we remove each vertex with degree less than $\ell$ in $G_1$ to obtain $G_2$. Then we remove each vertex with degree less than $\ell$ in $G_2$, and so on. At the end we obtain a graph $G'$ with minimum degree at least $\ell$. Then $G'$ is non-empty, as we deleted less than $\ell n$ edges from $E_1$. Thus there are at least $\ell$ vertices in $G'$.
 
 Let us now pick $v_1v_2$ arbitrarily from $E(G')$. Then there is an edge that can be picked as $v_{k-1}v_k$ so that the resulting $P_k$ does not align with any $P_\ell$, we pick such an edge. With $v_3,\dots,v_{k-2}$, they form a $P_k$ that does not align with any $P_\ell$. We will extend this $P_k$ to a $P_\ell$, obtaining a contradiction. Observe that $v_1$ has at least $\ell$ neighbors in $G'$, we pick one not on the path so far. We can extend the path to a $P_\ell$ the same way, by picking a neighbor of the vertex previously picked: the penultimate vertex still has at least one neighbor in $G'$ that is not on the path so far, thus we can finish the path to obtain a $P_\ell$, finishing the proof with a contradiction.
\end{proof}

%más fákra: levágtunk $a_i$ példányt $T_i$-ből és $b_i$ helyre tehetnénk $T_i$-t hogy visszakapjuk $T$-ből $T'$-t. Akkor mindent felfújunk, de valamalyik $i$-re a $b_i$ csúcs közül csak $a_i-1$-nek a szomszédjait, a többinek a szomszédjait nem. Pl csillag vs csillag lin, csillag vs csillagról lelógó él: csak $n$-et veszítünk, mert elég a középpontot nem felfújni.

Let us continue with the proof of Theorem \ref{zyk2}, that we restate here for convenience.

\begin{theorem*}
Let $H_1,\dots,H_\ell,H'_1,\dots,H'_m$ be complete multipartite graphs, $\alpha_i\ge 0$ 
%unless $H_i$ is a clique 
and $\beta_j\ge 0$.
%unless $H'_j$ is a clique. 
Then among $n$-vertex graphs $G$, $x(G):=\sum_{i=1}^\ell \alpha_i\cN_{\re}(H_i,K_{k_i+1},G)+\sum_{j=1}^m\beta_j \cN(H'_j,G)$ is maximized by a complete multipartite graph $T$. 
%Moreover, if hmm, ezt már nehéz ilyen általánosan, vegyük a legnagyobb $x$ pontszámút aminek pozitív az együtthatója, és $n^{x-2}$ amit bukunk ha nem complete multipartite...
\end{theorem*}

%We remark that $r$ does not have to be equal to $k$. Consider again $\ex_{\re}(n,K_{1,1,2},K_4)$, which we have shown to be equal to $i(K_{1,1,2},n)$. Brown and Sidorenko \cite{bs} showed that $i(K_{1,1,2},n)$ is obtained at a complete multipartite graph, but $T_5(n)$ contains more induced copies of $K_{1,1,2}$ than any complete 3-partite graph.

\begin{proof} Let $G$ be a graph on $n$ vertices, and for a vertex $v$ of $G$, let $d^*(v)$ denote the sum of the weights of copies of $H_i$ and $H'_j$ that are counted towards $x(G)$ and contain $v$. More precisely, let $d_i(v)$ denote the number of copies of $H_i$ in $G$ that do not align with any $K_{k_i+1}$ and contain $v$, let $d'_j(v)$ denote the number of copies of $H_j'$ in $G$ that contain $v$, and let $d^*(v)=\sum_{i=1}^\ell \alpha_i d_i(v)+\sum_{j=1}^m \beta_j d'_j(v)$. Similarly, let $d^*(u,v)$ denote the sum of the weights of copies of $H_i$ and $H'_j$ that are counted towards $x(G)$ and contain both $u$ and $v$.

We will apply Zykov's symmetrization. Consider two non-adjacent vertices $u$ and $v$ and assume without loss of generality that $d^*(u)\le d^*(v)$. Then we ``symmetrize'' $u$ to $v$: we remove all the edges incident to $u$, and for each edge $vw$, we add the edge $uw$ to obtain another graph $G'$.

\begin{clm} We have $x(G')\ge x(G)$. Moreover, if $d^*(u)<d^*(v)$, then $x(G')> x(G)$.
\end{clm}

\begin{proof}

Let $G_0$ be the graph we obtain by deleting the edges incident to $u$. We have $x(G_0)\ge x(G)-d^*(u)$. Indeed, $d^*(u)$ accounts for removing the nice copies of $H_i$ containing $u$ and for removing the copies of $H_j'$ containing $u$. It is possible that there are some other copies of $H_i$ in $G$ that align with some cliques $K_{k_i+1}$, but those cliques all contain $u$; in that case they become nice after deleting the edges incident to $u$, thus $x(G)$ may decrease by less than $d^*(u)$.

We will show that $x(G')\ge x(G_0)+d^*(u)$. First we will show that $d^*(u,v)$ does not decrease when we symmetrize $u$ to $v$.

Consider a nice copy $H_i^*$ of $H_i$ in $G$ that contains both $u$ and $v$. Then $u$ and $v$ belong to the same partite set of $H_i^*$. Thus $H_i^*$ is in $G'$, as the common neighbors of $u$ and $v$ in $G$ are also common neighbors of $u$ and $v$ in $G'$. 
If $H_i^*$ aligns with a copy $K_{k_i+1}^*$ of $K_{k_i+1}$ in $G'$, then $K_{k_i+1}^*$ cannot be in $G$, thus has to contain $u$ (the only vertex incident to new edges), but then $K_{k_i+1}^*$ cannot contain $v$. Therefore, we can replace $u$ with $v$ in $K_{k_i+1}^*$ we obtain a copy of $K_{k_i+1}$ in $G$ that aligns with $H_i^*$, a contradiction. This shows that the number of copies of $H_i$ aligning with no $K_{k_i+1}$ and containing both $u$ and $v$ does not decrease. 

Consider now a copy of $H_j'$ in $G$ that contains both $u$ and $v$. Then this copy is also in $G'$, thus the number of copies of $H'_j$ containing both $u$ and $v$ does not decrease. %Therefore, $d^*(u,v)$ does not decrease.

This shows that when comparing $x(G')$ to $x(G_0)$, we have an increase by $d^*(u,v)$ when counting only those copies of $H_i$ and $H_j$ that contain both $u$ and $v$. We will show that there is an increase of $d^*(v)-d^*(u,v)$ when counting the copies of $H_i$ and $H_j'$ that contain $u$ but not $v$. 
%(de bukunk-e valamit? ami nem tartalmazza $u$-t, de most már alignol valamihez. Nem, mert abban kell $u$-nak lennie, de akkor $v$-vel helyettesíthetnénk. Ezt a végére...)

Let us consider a nice copy $H_i^*$ in $G$ that contains $v$ but not $u$. Let $H_i^{**}$ be obtained from $H_i^*$ by replacing $v$ with $u$, then $H_i^{**}$ is in $G'$. If $H_i^{**}$ aligns with a copy $K_{k_i+1}^{**}$ of $K_{k_i+1}$ in $G'$, then $K_{k_i+1}^{**}$ does not contain $v$. Indeed, if $|V(H_i)|\ge k_i$, then $K_{k_i+1}^{**}$ is a subgraph of $H_i^{**}$, and if $|V(H_i)|< k_i$, then $H_i^{**}$ is a subgraph of $K_{k_i+1}^{**}$, thus $K_{k_i+1}^{**}$ contains $u$, hence it cannot contain $v$. Let $K_{k_i+1}^{*}$ be obtained from $K_{k_i+1}^{**}$ by replacing $u$ with $v$. Then $K_{k_i+1}^{*}$ is in $G$ and aligns with $H_i^*$, a contradiction. This means that for each copy of $H_i$ contributing to $d^*(v)-d^*(u,v)$, we have a copy of $H_i$ in $G'$ containing $u$ but not $v$ and contributing to $x(G')$. Similarly, if $H'_j$ contains $v$ but not $u$, then replacing $v$ with $u$ creates a copy of $H'_j$ in $G'$. This shows that the total weight of nice copies of $H_i$ and copies $H_j'$ that contain $u$ and not $v$ in $G'$ is at least $d^*(v)-d^*(u,v)$.

We have shown that the total weight of nice copies of $H_i$ and copies of $H'_j$ containing $u$ does not decrease, moreover, it increases if $d^*(v)>d^*(u)$. We also have to deal with copies of $H_i$ not containing $u$. They obviously are also in $G'$, but 
they could stop being nice, which would decrease $x(G)$. However, if such a copy $H_i^*$ aligns with a copy $K_{k_i+1}^{*}$ of $K_{k_i+1}$ in $G'$, then $K_{k_i+1}^{*}$ cannot be in $G$, thus has to contain $u$. This implies that $K_{k_i+1}^{*}$ does not contain $v$. Let $K_{k_i+1}^{**}$ be obtained from $K_{k_i+1}^{*}$ by replacing $u$ with $v$. then $K_{k_i+1}^{**}$ is in $G$ and aligns with $H_i^*$ in $G$, a contradiction.
%Ja még negatív súlyokkal is kell kezdeni valamit... vagy kidobni...
\end{proof}

We choose an arbitrary ordering $v_1,\dots,v_n$ of the vertices of $G$.
We will apply such symmetrization steps repeatedly, as long as there are two non-adjacent vertices $v_i$ and $v_j$ such that their neighborhood is not exactly the same, i.e. the symmetrization changes the graph. If $d^*(v_i)<d^*(v_j)$, then we symmetrize $v_i$ to $v_j$; if $d^*(v_i)>d^*(v_j)$, then we symmetrize $v_j$ to $v_i$. If $d^*(v_i)=d^*(v_j)$, then we check the smallest index $i'$ of vertices with the exact same neighborhood as $v_i$, and the smallest index $j'$ of vertices with the exact same neighborhood as $v_j$.
We
symmetrize $v_i$ to $v_j$ if $i'>j'$ and we symmetrize $v_j$ to $v_i$ if $i'<j'$.

\begin{clm}
There are finitely many symmetrization steps.
\end{clm}

\begin{proof}
 Steps where $d^*(v)\neq d^*(u)$ can happen at most $(\ell+m)n!2^n$ times, as the number of copies of some $H_i$ or $H'_j$ increases, but that number is bounded by $n!2^n$ for every $i$ and $j$. 
 %either the number of copies of some $H_i$ or $H'_j$ with positive coefficient increases, or the number of copies of some $H_i$ or $H'_j$ with negative coefficient decreases, but that number is bounded by $2^n$ for every $i$ and $j$. 
 
Between two such steps, $v_i$ is symmetrized to some other vertex at most $n2^{i-2}$ times. This can be shown by induction on $i$; for $i=1$, observe that $v_1$ is not symmetrized. For $i=2$, $v_2$ is only symmetrized to a vertex that has the same neighborhood as $v_1$, thus from that point $v_2$ has the same neighborhood as $v_1$. Indeed, if some other vertex $v_i$ is symmetrized to $v_j$, then either $v_j$ is connected to both $v_1$ and $v_2$ or none of them, thus this will hold for $v_i$ after the symmetrization.

If the statement holds for $i$, observe that $v_{i+1}$ can only be symmetrized only to some $v_j$ with $j\le i$ (or a vertex with the same neigborhood). However, $v_{i+1}$ may not have the same neighborhood as $v_j$ after a step where $v_j$ is symmetrized to another vertex. In this case we may symmetrize $v_{i+1}$ again to $v_j$ (or a vertex with the exact same neighborhood as $v_j$ at that point). This means that $v_{i+1}$ is symmetrized to $v_j$ (or a vertex with the exact same neighborhood as $v_j$ at that point) again only after $v_j$ was symmetrized to another vertex, thus at most $1+n2^{j-2}$ times by induction. Therefore, altogether $v_{i+1}$ is symmetrized to other vertices at most $\sum_{j=1}^{i}1+n2^{j-2}\le n2^{i-1}$ times, finishing the proof.
\end{proof}

After finitely many steps, we arrive to a graph $G^*$ that cannot be changed by symmetrization, thus non-adjacent vertices have the same neighborhood, which means being non-adjacent is an equivalence relation, i.e. $G^*$ is complete multipartite. 
%Let us consider now the moreover part. Let $G_0$ be the penultimate graph in this process, ... hmm, ugye $\alpha_i$ is függ $n$-től, hogy van a nagyságrend?
\end{proof}

Let us continue with the proof of Theorem \ref{getu}, that we restate here for convenience.
\begin{theorem*}
Let $F$ be a 3-chromatic graph with a color-critical edge and $a,b\ge 2$. Then for large enough $n$ we have that
$\ex(n,K_{a,b},F)=\ex(n,K_{a,b},K_3)=\cN(K_{a,b},K_{t,n-t})$ for some $t$.
\end{theorem*}
%\begin{thm} Let $F$ be a 3-chromatic graph with a color-critical edge. Then$\ex(n,K_{a,b},F)=\ex(n,K_{a,b},K_3)$ for large enough $n$.\end{thm}

\begin{proof}
 Let $G$ be an $n$-vertex $F$-free graph. Let us denote by $B_t$ the book graph, which consists of $t$ triangles sharing an edge. Let $B_t$ be the largest book in $G$. We separate two cases. 
 
 \textbf{Case 1.} $t=o(n)$.
 
 For any given triangle in $G$, there are at most $\alpha:=3\binom{t-1}{a-1}\binom{n}{b-2}+3\binom{t-1}{b-1}\binom{n}{a-2}+\binom{t-1}{a+b-3}=o(n^{a+b-3})$ copies of $K_{a,b}$ that align with that triangle. Indeed, either there are two vertices of the triangle in a partite set of $K_{a,b}$, and the third is in the other partite set, in which case we need to pick the other vertices of that partite set from a book containing the triangle, or all the three vertices of the triangle is on the same partite set of $K_{a,b}$, and we need to pick the $a+b-3$ other vertices from a book containing that triangle.
 
 Therefore, we have that $\cN(K_{a,b},G)\le\cN_{\re}(K_{a,b},K_3,G)+\alpha\cN(K_{a,b},G)$. By Theorem \ref{zyk2}, we have $\cN_{\re}(K_{a,b},K_3,G)+\alpha\cN(K_{a,b},G)\le\cN_{\re}(K_{a,b},K_3,T)+\alpha\cN(K_3,T)$ for a complete multipartite graph $T$. If $T$ has more than two partite sets, then Theorem \ref{brsi} implies that $\cN_{\re}(K_{a,b},K_3,T)$ is smaller by $\Omega(n^{a+b})$ than $\cN_{\re}(K_{a,b},K_3,T')$ for some complete bipartite graph $T'$, using that $\cN_{\re}(K_{a,b},K_3,G)$ is the number of induced copies of $K_{a,b}$ in $G$. As $\alpha\cN(K_3,T)=o(n^{a+b})$, this shows that $\cN(K_{a,b},G)\le\cN_{\re}(K_{a,b},K_3,G)+\alpha\cN(K_{a,b},G)\le\cN_{\re}(K_{a,b},K_3,T)+\alpha\cN(K_{a,b},T)\le \cN_{\re}(K_{a,b},K_3,T')-\Omega(n^{a+b})+\alpha\cN(K_{a,b},T')<\cN_{\re}(K_{a,b},K_3,T')=\cN(K_{a,b},T')$. If $T$ has at most two partite sets, then $\cN(K_3,T)=0$, thus we have $\cN(K_{a,b},G)\le \cN(K_{a,b},T)$.
 
 \textbf{Case 2.} $t=\Theta(n)$. 
 
 As $F$ is 3-chromatic with a color-critical edge, for some $k$ we have that $F$ is contained in the graph we obtain from $K_{k,k+2}$ by adding an additional edge inside the larger partite set.
  Consider a copy of $B_t$ in $G$, let $x$ and $y$ be the vertices of $B_t$ that are connected to each other vertex of that $B_t$ and let $U$ be the set of the other $t$ vertices in that book. Let $U'$ denote the set of $n-t-2$ vertices of $G$ not in the book. It is easy that there is no copy of $K_{k,k}$ such that one partite set is in $U$ and the other partite set is in $U\cup U'$, as that would form $K_{k,k+2}$ with $x$ and $y$. Then by the K\H ov\'ari-T.S\'os-Tur\'an theorem \cite{kst}, there are $o(n^2)$ edges between $U$ and $U'$ and inside $U$. 
  
  This shows that there are $o(n^{a+b})$ copies of $K_{a,b}$ containing an edge between $U$ and $U'$ or inside $U$. If $n-t=o(n)$, then there are also $o(n^{a+b})$ copies of $K_{a,b}$ not inside $U$, thus we are done.
  
  Assume now that $n-t=\Theta(n)$. The copies of $K_{a,b}$ are either inside $U$ (at most $\ex(t,K_{a,b},F)$ copies) or inside $U'$ ((at most $\ex(n-t-2,K_{a,b},F)$ copies) or contain $x$ or $y$ or an edge between $U$ and $U'$ ($o(n^{a+b})$ copies). Recall that we have $\ex(m,K_{a,b},F)=\cN(K_{a,b},T)+o(m^{a+b})$ for some $m$-vertex complete bipartite graph $T$ using Theorem \ref{gepa} and that $\ex(m,K_{a,b},K_3)=\cN(K_{a,b},T)$ for some complete bipartite graph (a result of Gy\H ori, Pach and Simonovits, mentioned at the introduction).
  %s are weakly 3-Tur\'an-good. 
  Therefore, we have $\cN(K_{a,b},G)=\cN(K_{a,b},T_1)+\cN(K_{a,b},T_2)+o(n^{a+b})$, where $T_1$ is a complete bipartite graph on $t$ vertices and $T_2$ is a complete bipartite graph on $n-t-2$ vertices. On the other hand, let us take a copy of $T_1$ and a copy of $T_2$ and connect the vertices of the larger partite sets of them. The resulting graph is bipartite, thus $F$-free. Clearly we created $\Theta(n^{a+b})$ additional copies of $K_{a,b}$, thus this graph contains more copies of $K_{a,b}$ than $G$.
 %Let us consider $K_{k,k+2}$ with an additional edge inside the larger partite set. For some $k$, this graph contains $F$.
 %Consider a copy of $B_t$ in $G$ and let $U$ be the set of $t$ vertices in that copy that are not contained in each triangle. Let $U'$ denote the set of $n-t-2$ vertices not in the book. Then there is no copy of $K_{k,k}$ wuch that one partite set is in $U$ and the other partite set is in $U'$. Then by the K\H ov\'ari-T.S\'os-Tur\'an theorem \cite{}, there are $o(n^2)$ edges between $U$ and $U'$.... 
\end{proof}

Let us continue with the proof of Proposition \ref{korok} that we restate here for convenience.

\begin{proposition*} $\ex_{\re}(n,K_3,C_{2k+1})=O(n^{1+1/k})$.
\end{proposition*}

\begin{proof} For the upper bound, we follow a proof F\"uredi and \"Ozkahya \cite{furozk}, who gave an upper bound on $\ex(n,K_3,C_{2k+1})$. 
First we claim that the number of triangles not aligning with any $C_\ell$ in an arbitrary graph $G$  is at most $(\ell-3)|E(G)|/3$. Indeed, consider a vertex $v$. The number of triangles containing $v$ is equal to the number of edges in its neighborhood. If there is a $P_{\ell-1}$ in the neighborhood of $v$, then the triangles formed by the edges of this path with $v$ are not nice. Therefore, the number of nice triangles containing $v$ is at most the number of edges in the graph $G'$, which we obtain by restricting $G$ to the neighborhood of $v$ and deleting all the edges of paths $P_{\ell-1}$. Then $G'$ is a $P_{\ell-1}$-free graph on $d(v)$ vertices, thus has at most $(\ell-3)d(v)/2$ edges by a well-known theorem of Erd\H os and Gallai \cite{Er-Ga}. Adding up for each vertex, we obtain an upper bound $(\ell-3)|E(G)|$, where each triangle is counted three times.

Let us consider now an arbitrary graph $G$, and take a random 3-coloring of its vertices, where every vertex gets color red, blue or green with probability 1/3. Let $G^*$ be obtained from $G$ by deleting the monochromatic edges, and afterwards deleting each edge that is not contained in any 3-chromatic nice triangle. Clearly, every triangle of $G$ is in $G'$ with probability at least $2/9$. Therefore, there is a 3-coloring such that at least 2/9 of the nice triangles of $G$ is in $G^*$. Obviously they are also nice triangles in $G^*$, thus it is enough to show that there are $O(n^{1+1/k})$ nice triangles in $G^*$. 

Consider first the subgraph $G_1$ of $G^*$ having the edges with blue and red end-vertices. We claim that $G_1$ does not contain $C_{2k}$. Indeed, let $uv$ be an edge of such a $C_{2k}$. Then $uv$ is in a 3-chromatic nice triangle $uvw$, thus $w$ is green. This means $w$ is not in the $C_{2k}$, thus replacing $uv$ with $uw$ and $wv$ we obtain a $C_{2k+1}$ in $G$ that aligns with the triangle $uvw$, a contradiction. Using the Bondy-Simonovits theorem \cite{bosi}, this implies that $G_1$ has $O(n^{1+1/k})$ edges. By the same argument for the colors blue and green and the colors green and red we obtain that $G^*$ has $O(n^{1+1/k})$ edges. By the first paragraph of the proof, we also have that $G^*$ contains $O(n^{1+1/k})$ nice triangles, completing the proof.
%By the above paragraph, it is enough to show that there are $O(n^{1+1/k})$ edges in $G'$. It is enough to show that for each color pair, there are $O(n^{1+1/k})$ edges between the vertices of those colors. We will show that there is no $C_{2k}$ with vertices having colors, say, blue and red. This will finish the proof, as
\end{proof}

Let us continue with the proof of Theorem \ref{parkor}, that we restate here for convenience.

\begin{theorem*}
 \textbf{(i)} $n^{2-o(1)}=\ex_{\re}(n,K_3,C_4)=o(n^2)$.

\textbf{(ii)} For any $k\ge 3$ we have that
$\ex_{\re}(n,K_3,C_{2k})=O(n^{1+\frac{1}{k-1}})$ and $\ex_{\re}(n,K_3,C_{2k})=\Omega(\ex(n,\cC_{2k-2}))$.
\end{theorem*}

\begin{proof}
 Let us start with the proof of \textbf{(i)} and let $B_2$ denote the book with 2 pages, i.e. two triangles sharing an edge (in other words, the 4-vertex graph with 5 edges). Alon and Shikhelman \cite{as} showed that $n^{2-o(1)}=\ex(n,K_3,B_2)=o(n^2)$. Observe that if an edge $uv$ of a graph is contained in at least two triangles, then each triangle containing $uv$ aligns with a $B_2$. Thus we can delete every such edge without decreasing the number of triangles that do not align with any copy of $B_2$. This way we obtain a graph where every edge is contained in at most one triangle, i.e. a $B_2$-free graph. This shows that $\ex_{\re}(n,K_3,B_2)=\ex(n,K_3,B_2)$. A triangle aligns with a $C_4$ if and only if it aligns with a $B_2$, thus we have $\ex_{\re}(n,K_3,C_4)=\ex_{\re}(n,K_3,B_2)=\ex(n,K_3,B_2)$. 
 
 Let us continue with the proof of \textbf{(ii)}. For the upper bound, we proceed similarly to the proof of Proposition \ref{korok}, thus from now on we assume familiarity with that proof and avoid repeating some details. We take an arbitrary graph $G$ and a random 3-coloring of its vertices, where every vertex gets color red, blue or green with probability 1/3. Let $G^*$ be obtained from $G$ by deleting the monochromatic edges, and afterwards deleting each edge that is not contained in any 3-chromatic nice triangle. Then, as in Proposition \ref{korok}, it is enough to show that there are $O(n^{1+\frac{1}{k-1}})$ nice triangles in $G_1$, which consists of the edges with blue and red end-vertices. 
 
 Consider the copies of $C_{2k-2}$ in $G_1$. For any edge $e=uv$ of it, there is a green vertex $w(e)$ connected to both $u$ and $v$ such that $uvw(e)$ is a nice triangle in $G$. If there are two edges $e$ and $e'$ in a copy of $C_{2k-2}$ such that $w(e)\neq w(e')$, then we can use both $w(e)$ and $w'(e)$ to extend the cycle and find a $C_{2k}$ containing $u,v,w(e)$, thus aligning with a triangle, a contradiction. Thus we have that each vertex of that $C_{2k-2}$ is connected to the same $w$, i.e. we found a wheel in $G$ %(and in $G^*$) 
 with center $w$, such that each triangle in the wheel is nice in $G$.
 
\begin{clm}
  Any two copies of $C_{2k-2}$ in $G_1$ with different vertex sets share at most one vertex.
\end{clm} 
\begin{proof}[Proof of Claim]
Assume indirectly that copies $C$ and $C'$ of $C_{2k-2}$ both contain $u$ and $v$. Let $w$ (resp. $w'$) be the centers of the wheels corresponding to $C$ and $C'$. Assume first that $w\neq w'$. Observe that we can go from $u$ to $v$ on a path using all the vertices of $C$ and $w$. With $w'$, this extends to a $C_{2k}$ in $G$ that aligns with the triangles in $C$, a contradiction. Assume now that $w=w'$. Then there is a vertex contained by both $C$ and $C'$, say $u$ that is connected to a vertex $u'$ of $C'$ that is not in $C$. Then we can go through the vertices of $C$ ending with $u$, then to $u'$ and then to $w$ to obtain a $C_{2k}$ that aligns with the triangles of $C$, a contradiction. 
\end{proof}
 
The above claim implies that there are $O(n^2)$ copies of $C_{2k-2}$ in $G''$. Indeed, any pair of vertices can belong to the vertex set of only one such copy, thus there are at most $\binom{n}{2}$ such vertex sets, and there are at most $(2k-3)!/2$ copies of $C_{2k-2}$ on a set of $2k-2$ vertices. We will use a supersaturation result of Simonovits, that appeared in \cite{ersi} without proof (see \cite{jiye} for a proof). It claims that there exist constants $c$ and $c'$ such that if we have $|E(H)|\ge c n^{1+1/k}$ for a graph $H$, then $H$ contains at least $c'\left(\frac{|E(H)|}{|V(H)|}\right)^{2k}$ copies of $C_{2k}$. 
%By choosing a large enough $c$, we can assume that $c'>(2k-3)!$. Indeed, if $|E(H)|= 2 c n^{1+1/k}$, then we can decompose $H$ to two graphs, both with at least $cn^{1+1/k}$ edges, thus we have at least $c'n^2$ edges. hmm, ez nem jó...
%Clearly if $c$ increases then $c'$ must increase csak vegyük szét két részre, $H$-t, mindkettőben van sok kör
Applying it with $G_1$ and $C_{2k-2}$, we obtain that $|E(G_1)|\le cn^{1+\frac{1}{k-1}}$ for a large enough $c$.

From this point we can again follow the proof of Proposition \ref{korok}: $|E(G^*)|\le 3cn^{1+\frac{1}{k-1}}$, thus the number of nice triangles in $G^*$ is $O(n^{1+\frac{1}{k-1}})$, which implies that the number of nice triangles in $G$ is $O(n^{1+\frac{1}{k-1}})$, completing the proof of the upper bound.

For the lower bound, consider an $n$-vertex bipartite $\cC_{2k-2}$-free graph $G$ with $\Omega(\ex(n,\cC_{2k-2}))$ edges, and let us double the vertices on one side. It means that we replace each vertex $v$ with two vertices $v_1,v_2$, and for each edge $uv$, we add the edges $uv_1$ and $uv_2$. Let $G'$ be the graph obtained this way, then clearly $G'$ has $|E(G')|$ triangles. We claim that each of those is nice. Observe that a cycle in $G'$ corresponds to a walk in $G$ if we consider the visits to $v_1$ and $v_2$ as visiting $v$ twice. A copy $C$ of $C_{2k}$ corresponds to a walk that has $2k$ vertices with repetitions. As there are no cycles of length less than $2k$ in $G$, there are only two possibilities. One is that the walk is on a single edge in $G$, but then the corresponding cycle has at most 4 vertices in $G'$, a contradiction. The other possibility is that the walk in $G$ is a $C_{2k}$, i.e. $C$ does not visit both vertices corresponding to a vertex of $G$. But each triangle $T$ contains two vertices corresponding to a vertex of $G$, thus $T$ does not align with $C$, completing the proof. 
\end{proof}

\section{Concluding remarks}\label{concl}

%Ja várjunk! Nem is kell ez! Egyszerűen minden $K_k$-ból kitörlünk egy élt, akkor kevés $H$-t töröltünk ki, maradékra kell stability. nem teljesen, mert lehet hogy maradék már $k-1$ osztályú. De ebben az esetben ezt is meg lehet majd oldani: ha ilyen, akkor kevés él hiányzik, akkor nem lehet (sok) plusz él. Ja nem jó, ott éleket törlünk, akkor sok $K$-t törlünk

$\bullet$ Let us return to $\ex_{\re}(n,P_k,P_\ell)$. If $0<\ell-k$ is even, then the upper bound of Proposition \ref{utak} still holds, but we can improve it. We briefly give a sketch of the arguments. Let $H$ be the graph consisting of $k$ internally vertex disjoint paths of length $\ell-k$ connecting two vertices $u$ and $v$. This graph is called a theta graph and it is well-known \cite{fasi} that $\ex(n,H)=O(n^{1+1/(\ell-k)})$. We pick copies $v_1v_2\dots v_k$ of $P_k$ not aligning with any $P_\ell$ by picking first the edge $v_1v_2$, then $v_3v_4$, and so on. One can show that each time, the possible edges to pick do not contain $H$, giving the upper bound $O(n^{k/2+k/(2(\ell-k)})$ if $k$ is even, and $O(n^{(k+1)/2+(k-1)/(2(\ell-k)})$ if $k$ is odd.

We also have the lower bound $\ex_{\re}(n,P_k,P_\ell)\ge \ex(n,P_k,P_\ell)=\Theta(n^{\lceil k/2\rceil})$. We can improve it in the case $k>4$ is even. Let us again give a sketch. One of the constructions for $\ex(n,P_k,P_\ell)$ goes by blowing up every second vertex of $k$-vertex path $v_1v_2\dots v_k$ to linear size. In the case $k$ is even, we have some freedom here, as we need to only avoid blowing up two adjacent vertices, thus we can blow up $v_1,v_4,v_6,\dots,v_k$. For our case, we modify it slightly: we blow up $v_3$ as well, but instead of placing a complete bipartite graph between the vertices replacing $v_3$ and $v_4$, we place a bipartite graph $G'$ with girth at least $\ell-k$. We consider the copies of $P_k$ where the $i$th vertex is $v_i$ or a replacement of $v_i$. Then a copy of a longer path containing every vertex of such a path has a circle inside $G'$, but this means more than $\ell-k$ additional vertices. It is easy to see that this way we get $\Omega(|E(G')|n^{k/2-1})$ copies of $P_k$ not aligning with any $P_\ell$. This gives a lower bound increased by a factor of roughly $n^{2/3(\ell-k)}$, using a theorem of Lazebnik, Ustimenko and Woldar \cite{luw}.

%For the lower bound, we replace $v_i$ by $\lfloor n/k\rfloor$ new vertices unless $i$ is divisible by 3. Then we connect every $v_{3j}$ to every vertex replacing its neighbors. Finally, between vertices replacing $v_{3j+1}$ and $v_{3j+2}$, we place a graph $G_j$ with girth at least $\ell-k$. We consider the copies of $P_k$ where the $i$th vertex is $v_i$ or a replacement of $v_i$. Then a copy of a longer path containing every vertex of such a path has a circle inside a $G_j$, but this means more than $\ell-k$ vertices. Ja ez nem jó, mármint nem sok él... akkor lehet hogy mégis az alsó korlát az éles...

\smallskip

$\bullet$ One could examine the following question.
If we are given an integer $n$ and graphs $F$, $F'$ and $H$, at most how many copies of $H$ can align with some $F$ in an $n$-vertex $F'$-free graph? Clearly, this number plus $\ex_{\re}(n,H,F)$ is an upper bound on $\ex(n,H,F')$. We have examined some problems where we combined a trivial upper bound on this number and a non-trivial bound on $\ex_{\re}(n,H,F)$.

Let us show an example where the opposite happens. More precisely, we have one more twist: we consider copies of $H$ aligning with any copy of any graph from a family of graphs. Gerbner, Gy\H ori, Methuku and Vizer \cite{GGMV2020} showed that $\ex(n,C_4,C_{2k})=(1+o(1))\binom{k-1}{2}\binom{n}{2}$ using the following idea. We say that a pair of vertices is \emph{fat} if they have at least $k$ common neighbors. We count separately the 4-cycles having two fat pairs of opposite vertices and the 4-cycles having at most one fat pair of opposite vertices. They showed that there are $O(n^{1+1/k})$ 4-cycles of the first type and at most $(1+o(1))\binom{k-1}{2}\binom{n}{2}$ of the second type. This second statement is a simple observation: we pick two vertices $\binom{n}{2}$ ways to be a non-fat pair of opposite vertices and then we have to pick the other two vertices out of their at most $k-1$ common neighbors. Observe that here we count the copies of $C_4$ that do not align with any graph obtained from a $C_4$ by adding $k-2$ new common neighbors to both opposite pairs (where the common neighbors of the two pairs may coincide).

$\bullet$ Let us mention an example where handling together the copies of $H$ that align with $F$ and those that do not align with $F$ gives more than our approach. 
It is a simple proof of Qian, Xie and Ge \cite{qxg}. Let $H$ be a graph obtained by adding a vertex to $K_\ell$ and connecting it to an arbitrary number of vertices. They showed that $\ex(n,H,K_k)=\cN(H,T_{k-1}(n))$ for any $k>\ell$. The number of copies of $H$ that align with $K_{\ell+1}$ is maximized by the Tur\'an graph using Zykov's theorem. However, it is not clear whether the number of copies of $H$ that do not align with $K_{\ell+1}$ is maximized by the Tur\'an graph. What helps is that the surplus in the number of copies of $H$ that align with $H$ can also be used.

%Let us consider a $K_k$-free graph $G$ on $n$ vertices. We first count the copies of $H'=K_\ell+K_1$. We first pick a $K_\ell$ (the number of ways is maximized by the Tur\'an graph), and then an additional vertex ($n-\ell$ ways). Let $a$ denote the number of copies of $H'$ that align with a copy of $K_{\ell+1}$ and $a'$ the number of other copies; let $b$ and $b'$ denote the analogous quantities in the Tur\'an graph.
%The above arguments show that we have $a+a'\le b+b'$, and $a\le b$. Then $a'\le b-a+b'$. This means that we can define an injection from the copies of $H'$ in $G$ to the copies of $H'$ in the Tur\'an graph such that a copy that aligns with a copy of $K_{\ell+1}$ in $G$ is mapped to a copy that aligns with a copy of $K_{\ell+1}$ in the Tur\'an graph. We can count the copies of $H$ by picking $H'$ first and then the remaining edges. The number of ways to pick the remaining edges increases with the number of edges in $G$ between the vertices of that copy of $H'$. If $H'$ does not align with $K_{\ell+1}$, there are at most $\ell-1$ such edges (in addition to the edges that belong to that copy of $H'$). It is mapped to a copy of $H'$ in the Tur\'an graph with at least as many additional edges. Similarly, if $H'$ aligns with a copy of $K_{\ell+1}$ in $G$, it is mapped to a copy with the same number of additional edges, finishing the proof.

$\bullet$
%Tulajdonképpen induced általánosítása complete multipartite esetén, mert ott eggyal nagyobb klikkhez nem alignol
Finally, let us discuss a variant. In the special case where $F$ contains $H$, we say that a copy $H^*$ of $H$ \emph{strongly aligns} with a copy $F^*$ of $F$ in $G$ if $F^*$ contains $H^*$ as a subgraph. It means that not only the vertices, but also the edges of $H^*$ are in $F^*$. For example, the analogue of Proposition \ref{utak} also holds in this setting, and immediately extends to the case $\ell-k$ is even. Indeed, in the proof of Proposition \ref{utak} we showed $\Omega(n^{k-2})$ copies of $P_k$ that cannot be extended to longer paths.

\end{document}